\documentclass[12pt,reqno]{article}

\usepackage[usenames]{color}
\usepackage{amssymb}
\usepackage{amsmath}
\usepackage{amsthm}
\usepackage{amsfonts}
\usepackage{amscd}
\usepackage{graphicx}
\usepackage{mathrsfs}

\usepackage[colorlinks=true,
linkcolor=webgreen,
filecolor=webbrown,
citecolor=webgreen]{hyperref}

\definecolor{webgreen}{rgb}{0,.5,0}
\definecolor{webbrown}{rgb}{.6,0,0}
\definecolor{red}{rgb}{1,0,0}

\usepackage{color}
\usepackage{fullpage}
\usepackage{float}

\usepackage{latexsym}
\usepackage{epsf}
\usepackage{breakurl}

\let\up=\textsuperscript
\let\epsilon=\varepsilon

\def\C{{\mathbb C}}
\def\N{{\mathbb N}}

\def\c{\mathrm{c}}

\def\idem{\leavevmode\hbox to 10.6mm{\vrule height .63ex depth -.59ex
    width 10mm\hfill}}

\newcommand{\pinf}[1]{^{\underline{#1}}}
\newcommand{\vabs}[1]{\left\vert{#1}\right\vert}

\begin{document}

\theoremstyle{plain}
\numberwithin{equation}{section}
\newtheorem{thm}{Theorem}[section]
\newtheorem{theorem}[thm]{Theorem}
\newtheorem{lemma}[thm]{Lemma}
\newtheorem{prop}[thm]{Proposition}
\newtheorem{coll}[thm]{Corollary}
\newtheorem{conj}[thm]{Conjecture}

\newtheorem{thmn}{Theorem}     
\newtheorem{colln}[thmn]{Corollary}  
\newtheorem{propn}{Proposition} 

\newtheorem*{thmnn}{Theorem}

\theoremstyle{definition}
\newtheorem{defi}[thm]{Definition}

\theoremstyle{remark}
\newtheorem{rmk}[thm]{Remark}
\newtheorem{rmks}[thm]{Remarks}
\newtheorem{expl}[thm]{Example}
\newtheorem{expls}[thm]{Examples}

\begin{center}
\vskip 1cm{\Large\bf New Formulas Involving Bernoulli and Stirling \\[3mm] Numbers of Both Kinds 
}
\vskip 1cm
\large
Bakir Farhi \\
National Higher School of Mathematics \\
P.O.Box 75, Mahelma 16093, Sidi Abdellah (Algiers) \\
Algeria \\
\href{mailto:bakir.farhi@nhsm.edu.dz}{\tt bakir.farhi@nhsm.edu.dz}
\end{center}

\vskip .2 in

\begin{abstract}
This paper is devoted to establishing several new formulas relating Bernoulli and Stirling numbers of both kinds.
\end{abstract}

\section{Introduction and Notation}\label{sec1}
Throughout this paper, we let $\N$ and $\N_0$ respectively denote the set of positive integers and the set of nonnegative integers. For $n \in \N_0$, we let $X\pinf{n}$ denote \textit{the falling factorial of $X$ to depth $n$}, that is 
$$
X\pinf{n} := X (X - 1) \cdots (X - n + 1) .
$$
The Stirling numbers of the first kind $s(n , k)$ ($n , k \in \N_0$, $n \geq k$) are then defined through the polynomial identity:
\begin{equation}\label{eq1}
X\pinf{n} = \sum_{k = 0}^{n} s(n , k) X^k ~~~~~~~~~~ (\forall n \in \N_0) ,
\end{equation}
while the Stirling numbers of the second kind $S(n , k)$ ($n , k \in \N_0$, $n \geq k$) are defined through the inverse identity:
\begin{equation}\label{eq2}
X^n = \sum_{k = 0}^{n} S(n , k) X\pinf{k} ~~~~~~~~~~ (\forall n \in \N_0) .
\end{equation}
For algebraic and combinatorial reasons, it is also considered that $s(n , k) = S(n , k) = 0$ for all $n , k \in \N_0$, with $n < k$. From \eqref{eq1} and \eqref{eq2}, we immediately derive the following recurrence formulas allowing to compute, step by step, the numbers $s(n , k)$ and $S(n , k)$:
\begin{align}
s(n , k) & = s(n - 1 , k - 1) - (n - 1) \, s(n - 1 , k) , \label{eq3} \\
S(n , k) & = S(n - 1 , k - 1) + k \, S(n - 1 , k) \label{eq4}
\end{align}
(valid for all $n , k \in \N$). By inserting \eqref{eq2} into \eqref{eq1} and conversely, we also obtain the orthogonality relations
\begin{align}
\sum_{m \leq i \leq n} s(n , i) S(i , m) & = \delta_{n m} , \label{eq5} \\
\sum_{m \leq i \leq n} S(n , i) s(i , m) & = \delta_{n m} \label{eq6}
\end{align}
(valid for all $n , m \in \N_0$ with $m \leq n$), where $\delta_{n m}$ denotes the Kronecker delta.  For additional formulas and combinatorial interpretations regarding Stirling numbers of both kinds, we refer to \cite{com,stan}. Further, by substituting in \eqref{eq1} $X$ by $(- X)$ and rearranging, we get (for all $n \in \N_0$)
$$
X (X + 1) \cdots (X + n - 1) = \sum_{k = 0}^{n} (-1)^{n - k} s(n , k) X^k ,
$$
implying that the integers $(-1)^{n - k} s(n , k)$ ($n , k \in \N_0$, $n \geq k$) are all nonnegative, that is the sign of $s(n , k)$ is $(-1)^{n - k}$. Thus the above polynomial identity is equivalent to
\begin{equation}\label{eq17}
X (X + 1) \cdots (X + n - 1) = \sum_{k = 0}^{n} \vabs{s(n , k)} X^k ~~~~~~~~~~ (\forall n \in \N_0) . 
\end{equation}
For $n \in \N$, dividing both sides of the above identity by $X$ and rearranging gives
\begin{equation}\label{eq18}
(X + 1) (X + 2) \cdots (X + n - 1) = \sum_{k = 0}^{n - 1} \vabs{s(n , k + 1)} X^k ~~~~~~~~~~ (\forall n \in \N) . 
\end{equation}

Next, the Bernoulli polynomials $B_n(X)$ ($n \in \N_0$) can be defined by their exponential generating function
\begin{equation}\label{eq7}
\frac{t e^{X t}}{e^t - 1} = \sum_{n = 0}^{+ \infty} B_n(X) \frac{t^n}{n!}
\end{equation}
and the Bernoulli numbers $B_n$ ($n \in \N_0$) are the values of the Bernoulli polynomials at $X = 0$; that is $B_n := B_n(0)$ ($\forall n \in \N_0$). To differentiate between the Bernoulli polynomials and the Bernoulli numbers, we always put the indeterminate $X$ in evidence when referring to polynomials. For a comprehensive and modern treatment of Bernoulli polynomials and numbers, see \cite{kou}. Similarly, the Bernoulli numbers of the second kind $B_n^*$ ($n \in \N_0$) are defined by their exponential generating function
\begin{equation}\label{eq8}
\frac{t}{\log(1 + t)} = \sum_{n = 0}^{+ \infty} B_n^* \, \frac{t^n}{n!} .
\end{equation}
A simple and useful formula for computing the $B_n^*$'s is given by \cite{rom}:
\begin{equation}\label{eq13}
B_n^* = \int_{0}^{1} x\pinf{n} \, d x ~~~~~~~~~~ (\forall n \in \N_0) .
\end{equation}
To avoid ambiguity, we often refer to Bernoulli numbers as ``Bernoulli numbers of the first kind'' in this paper. 

We frequently use the linear operators $I$, $\tau_r$ ($r \in \C$), $D$, and $\Delta$ of $\C[X]$, which respectively denote the identity of $\C[X]$, the translation by $r$, the differential operator ($D := \frac{d}{d \, X}$), and the forward difference operator ($\Delta := \tau_1 - I$). Applying the operators $D$ and $\Delta$ to certain types of polynomials yields remarkable and useful formulas, such as:
\begin{align}
\Delta X\pinf{n} & = n X\pinf{n - 1} , \label{eq9} \\
\Delta B_n(X) & = n X^{n - 1} , \label{eq10} \\
B_n'(X) & = n B_{n - 1}(X) \label{eq11}  
\end{align}
(valid for all $n \in \N$). Reiterating \eqref{eq11}, we derive that $B_n^{(k)}(0) = n\pinf{k} B_{n - k}$ ($\forall n , k \in \N_0$, with $k \leq n$). Hence, by the Taylor formula, we get
\begin{equation}\label{eq14}
B_n(X) = \sum_{k = 0}^{n} \binom{n}{k} B_{n - k} X^k ~~~~~~~~~~ (\forall n \in \N_0) .
\end{equation}
From \eqref{eq10} and \eqref{eq11}, we also derive that
\begin{equation}\label{eq12}
\int_{0}^{1} B_n(x) \, d x = 0 ~~~~~~~~~~ (\forall n \in \N) .
\end{equation} 
Additionally, all these operators can be expressed both in terms of $D$ and in terms of $\Delta$. Indeed, the Taylor formula shows that $\tau_r = e^{r D}$ ($\forall r \in \C$), thus $\Delta = \tau_1 - I = e^D - I$, implying that $D = \log(I + \Delta)$. The relationships between different linear operators play a vital role in exploring new identities involving special types of polynomials and numbers, such as those studied here. Although the operators $D$ and $\Delta$ do not have proper inverses, each has an inverse up to an additive constant. By abuse of notation, we let $D^{-1}$ (or $\int$) and $\Delta^{-1}$ respectively denote the inverses of $D$ and $\Delta$ (defined up to an additive constant). So, from Formulas \eqref{eq9} and \eqref{eq10}, we immediately derive the following formulas which hold up to an additive constant and for all $n \in \N_0$:
\begin{align}
\Delta^{-1} X\pinf{n} & = \frac{X\pinf{n + 1}}{n + 1} , \label{eq19} \\[1mm]
\Delta^{-1} X^n & = \frac{B_{n + 1}(X)}{n + 1} . \label{eq20}
\end{align}

The Bernoulli numbers were first introduced by Jacob Bernoulli in the late 17\up{th} century to establish a closed form for the sum of powers of consecutive natural numbers. Since then, they have become fundamental in various branches of mathematics, including number theory, mathematical analysis, and algebraic topology, among others.

The Bernoulli numbers of the second kind, introduced by James Gregory in 1670 and rediscovered several times by renowned mathematicians of the 19\up{th} century (see \cite{bla}), have recently reemerged as a subject of study (see e.g., \cite{nem,qz1,qz2}). However, it is worth noting that some authors define the Bernoulli numbers of the second kind as $\frac{B_n^*}{n!}$ rather than $B_n^*$. In this paper, we follow Roman's approach \cite[p. 113-114]{rom}.   

The goal of this paper is to establish new formulas relating Bernoulli and Stirling numbers of both kinds. Specifically, Corollary \ref{c1} and Theorem \ref{t2} provide formulas expressing Bernoulli numbers of the first kind in terms of Bernoulli numbers of the second kind, and vice versa. Many of our results come in pairs, where each result about Bernoulli numbers of the first kind has a corresponding result for Bernoulli numbers of the second kind. For instance, for Corollary \ref{c2}, Corollary \ref{c4}, and Theorem \ref{t5} correspond respectively Theorem \ref{t4}, Corollary \ref{c5}, and Theorem \ref{t6}. Remarkably, and up to some details, to a formula involving Bernoulli and Stirling numbers of both kinds, there corresponds a formula obtained by switching the roles of the two kinds of Bernoulli and Stirling numbers. The pairs of results that perfectly illustrate this phenomenon are Corollary \ref{c1} with Theorem \ref{t2} and Corollary \ref{c2} with Theorem \ref{t4}.

\section{The results and the proofs}\label{sec2}

We begin with the following theorem which provides, for a given $n \in \N_0$, the expression of the Bernoulli polynomial $B_n(X)$ as a linear combination of the polynomials $X\pinf{k}$ ($0 \leq k \leq n$).

\begin{thm}\label{t1}
For all $n \in \N_0$, we have
$$
B_n(X) = B_n + \sum_{k = 1}^{n} \frac{n}{k} S(n - 1 , k - 1) X\pinf{k} .
$$
\end{thm}

\begin{proof}
The identity of the theorem is clearly true for $n = 0$. Next, let $n \in \N$ be fixed. According to \eqref{eq10}, \eqref{eq2}, and \eqref{eq9}, we have
\begin{align*}
\Delta B_n(X) & = n X^{n - 1} \\
& = n \sum_{k = 0}^{n - 1} S(n - 1 , k) X\pinf{k} \\
& = n \sum_{k = 0}^{n - 1} \frac{S(n - 1 , k)}{k + 1} \Delta X\pinf{k + 1} \\
& = \Delta\left(n \sum_{k = 0}^{n - 1} \frac{S(n - 1 , k)}{k + 1} X\pinf{k + 1}\right) \\
& = \Delta\left(n \sum_{k = 1}^{n} \frac{S(n - 1 , k - 1)}{k} X\pinf{k}\right) . 
\end{align*}
Consequently, there exists a complex constant $\c$ satisfying the polynomial identity:
$$
B_n(X) = n \sum_{k = 1}^{n} \frac{S(n - 1 , k - 1)}{k} X\pinf{k} + \c .
$$
By taking in the last identity $X = 0$, we get $\c = B_n(0) = B_n$, confirming the required identity of the theorem.
\end{proof}

From Theorem \ref{t1}, we derive the following corollary which provides, for a given $n \in \N$, an expression of the Bernoulli number of the first kind $B_n$ in terms of the Bernoulli numbers of the second kind $B_k^*$ ($1 \leq k \leq n$) and Stirling numbers of the second kind.

\begin{coll}\label{c1}
For all $n \in \N$, we have
$$
B_n = - \sum_{k = 1}^{n} \frac{n}{k} S(n - 1 , k - 1) B_k^* .
$$
\end{coll}

\begin{proof}
It suffices to integrate from $0$ to $1$ the polynomial identity of Theorem \ref{t1} and use Formulas \eqref{eq12} and \eqref{eq13}.
\end{proof}

Now, we invert the formulas in Theorem \ref{t1} and Corollary \ref{c1}; that is, we express the polynomial $X\pinf{n}$ ($n \in \N_0$) as a linear combination of Bernoulli polynomials, and we express the Bernoulli number of the second kind $B_n^*$ ($n \in \N_0$) in terms of Bernoulli and Stirling numbers of the first kind. To do so, we need the following lemma.

\begin{lemma}\label{l1}
Let ${(u_n)}_{n \in \N_0}$ and ${(v_n)}_{n \in \N_0}$ be two complex sequences. Then the two following identities are equivalent:
\begin{align}
u_n & = \sum_{k = 0}^{n} s(n , k) v_k ~~~~~~~~~~ (\forall n \in \N_0) , \tag{$I$} \\
v_n & = \sum_{k = 0}^{n} S(n , k) u_k ~~~~~~~~~~ (\forall n \in \N_0) . \tag{$II$}
\end{align}
\end{lemma}

\begin{proof}
Use the orthogonality relations \eqref{eq5} and \eqref{eq6} (see e.g., \cite{com} or \cite{rio} for the details).
\end{proof}

By applying Lemma \ref{l1} for the formulas in Theorem \ref{t1} and Corollary \ref{c1}, we obtain the following results.

\begin{thm}\label{t2}
For all $n \in \N$, we have
$$
B_n^* = - \sum_{k = 1}^{n} \frac{n}{k} s(n - 1 , k - 1) B_k .
$$
\end{thm}

\begin{proof}
First, remark that the formula in Corollary \ref{c1} can be rewritten as
$$
\frac{B_{n + 1}}{n + 1} = \sum_{k = 0}^{n} S(n , k) \left(- \frac{B_{k + 1}^*}{k + 1}\right) ~~~~~~~~~~ (\forall n \in \N_0) ,
$$
and then apply Lemma \ref{l1} to invert this last formula.
\end{proof}

\begin{thm}\label{t3}
For all $n \in \N_0$, we have
$$
X\pinf{n} = B_n^* + \sum_{k = 1}^{n} \frac{n}{k} s(n - 1 , k - 1) B_k(X) .
$$
\end{thm}

\begin{proof}
First, remark that the identity in Theorem \ref{t1} can be rewritten as
$$
\frac{B_{n + 1}(X) - B_{n + 1}}{n + 1} = \sum_{k = 0}^{n} S(n , k) \left(\frac{X\pinf{k + 1}}{k + 1}\right) ~~~~~~~~~~ (\forall n \in \N_0) ,
$$
and then apply Lemma \ref{l1} to invert this last identity and use Theorem \ref{t2}.
\end{proof}

From Theorem \ref{t1}, we also derive a formula expressing the Bernoulli numbers in terms of Stirling numbers of both kinds, a result previously pointed out by Quaintance and Gould \cite{qg} and, more recently, by Cereceda \cite{cer}. We have the following corollary:

\begin{coll}\label{c2}
For all $n , k \in \N$, with $n \geq k$, we have
$$
\sum_{k \leq i \leq n} \frac{S(n - 1 , i - 1) s(i , k)}{i} = \frac{1}{n} \binom{n}{k} B_{n - k} .
$$
\end{coll}

\begin{proof}
Let $n \in \N$ be fixed. According to Theorem \ref{t1} and Formula \eqref{eq1}, we have that
\begin{align*}
B_n(X) & = B_n + \sum_{i = 1}^{n} \frac{n}{i} S(n - 1 , i - 1) X\pinf{i} \\
& = B_n + \sum_{i = 1}^{n} \frac{n}{i} S(n - 1 , i - 1) \sum_{k = 1}^{i} s(i , k) X^k \\
& = B_n + n \sum_{\begin{subarray}{c}
1 \leq i \leq n \\
1 \leq k \leq i
\end{subarray}} \frac{S(n - 1 , i - 1) s(i , k)}{i} X^k \\
& = B_n + n \sum_{\begin{subarray}{c}
1 \leq k \leq n \\
k \leq i \leq n
\end{subarray}} \frac{S(n - 1 , i - 1) s(i , k)}{i} X^k \\
& = B_n + n \sum_{k = 1}^{n} \left(\sum_{k \leq i \leq n} \frac{S(n - 1 , i - 1) s(i , k)}{i}\right) X^k . 
\end{align*}
By identifying this with \eqref{eq14}, we derive that for all $k \in \N$, with $k \leq n$, we have
$$
n  \sum_{k \leq i \leq n} \frac{S(n - 1 , i - 1) s(i , k)}{i} = \binom{n}{k} B_{n - k} ,
$$
confirming the formula of the corollary.
\end{proof}

\begin{rmk}
Since $s(n , k) = S(n , k) = 0$ for all $n , k \in \N_0$ with $n < k$, the summation condition ``$k \leq i \leq n$'' appearing in the formula of Corollary \ref{c2} can be replaced by the simpler one ``$i \in \N$''. 
\end{rmk}

From Corollary \ref{c2}, we immediately derive the following well-known result, which expresses the Bernoulli numbers in terms of Stirling numbers of the second kind. This result, in particular, served as the starting point for Carlitz \cite{car} in providing an interesting alternative proof of the famous Von Staudt-Clausen theorem.

\begin{coll}\label{c3}
For all $n \in \N_0$, we have
$$
B_n = \sum_{i = 0}^{n} (-1)^i \frac{i!}{i + 1} S(n , i) .
$$
\end{coll}

\begin{proof}
It suffices to apply Corollary \ref{c2} for $(n , k)$ replaced by $(n + 1 , 1)$ ($n \in \N_0$) and note that $s(m , 1) = (-1)^{m - 1} (m - 1)!$ ($\forall m \in \N$). 
\end{proof}

Further, by relying on the recurrence relation \eqref{eq4} and the orthogonality relation \eqref{eq6}, we derive from Corollary \ref{c2} the following corollary:

\begin{coll}\label{c4}
For all $n , k \in \N$, with $n \geq k$, we have
$$
\sum_{k \leq i \leq n} \frac{S(n , i) s(i , k)}{i} = \frac{1}{n} \binom{n}{k} B_{n - k} + \delta_{n - 1 , k} .
$$
\end{coll}

\begin{proof}
Let $n , k \in \N$ be fixed, with $n \geq k$. By the recurrence relation \eqref{eq4}, we have for all $i \in \{k , k + 1 , \dots , n\}$: $S(n , i) = S(n - 1 , i - 1) + i S(n - 1 , i)$, that is
$$
S(n - 1 , i - 1) = S(n , i) - i S(n - 1 , i) .
$$
By inserting this in the formula of Corollary \ref{c2} (for all $i \in \N$, with $k \leq i \leq n$), we get after simplifying and rearranging
\begin{align*}
\sum_{k \leq i \leq n} \frac{S(n , i) s(i , k)}{i} & = \frac{1}{n} \binom{n}{k} B_{n - k} + \sum_{k \leq i \leq n} S(n - 1 , i) s(i , k) \\
& = \frac{1}{n} \binom{n}{k} B_{n - k} + \delta_{n - 1 , k}
\end{align*}
(according to the orthogonality relation \eqref{eq6}). The corollary is proved.
\end{proof}

In what follows, we will establish analogs of the formulas in Corollaries \ref{c2} and \ref{c4}, where the kinds of Stirling numbers are permuted. In other words, we will find closed forms for the two sums $\sum_{k \leq i \leq n} \frac{s(n - 1 , i - 1) S(i , k)}{i}$ and $\sum_{k \leq i \leq n} \frac{s(n , i) S(i , k)}{i}$ ($n , k \in \N$, $n \geq k$). Interestingly, the resulting closed forms depend on Bernoulli numbers of the second kind (see Theorem \ref{t4} and Corollary \ref{c5} below).

\begin{thm}\label{t4}
For all $n , k \in \N$, with $n \geq k$, we have
$$
\sum_{k \leq i \leq n} \frac{s(n - 1 , i - 1) S(i , k)}{i} = \frac{1}{n} \binom{n}{k} B_{n - k}^* .
$$
\end{thm}

\begin{proof}
Let $n \in \N$ be fixed. According to Formula \eqref{eq1}, we have
$$
x\pinf{n - 1} = \sum_{i = 0}^{n - 1} s(n - 1 , i) x^i .
$$
By integrating with respect to $x$, we get (up to an additive constant)
\begin{align*}
\int x\pinf{n - 1} \, d x & = \sum_{i = 0}^{n - 1} \frac{s(n - 1 , i)}{i + 1} x^{i + 1} \\
& = \sum_{i = 0}^{n - 1} \frac{s(n - 1 , i)}{i + 1} \sum_{k = 1}^{i + 1} S(i + 1 , k) x\pinf{k} ~~~~~~~~~~ (\text{by Formula \eqref{eq2}}) \\
& = \sum_{k = 1}^{n} \left(\sum_{k - 1 \leq i \leq n - 1} \frac{s(n - 1 , i) S(i + 1 , k)}{i + 1}\right) x\pinf{k} ,
\end{align*}
that is (up to an additive constant)
\begin{equation}\label{eq15}
\int x\pinf{n - 1} \, d x = \sum_{k = 1}^{n} \left(\sum_{k \leq i \leq n} \frac{s(n - 1 , i - 1) S(i , k)}{i}\right) x\pinf{k} .
\end{equation}
On the other hand, we have (according to Formula \eqref{eq8})
$$
\frac{1}{\log(1 + t)} = \frac{1}{t} \sum_{k = 0}^{+ \infty} B_k^* \frac{t^k}{k!} = \frac{1}{t} + \sum_{k = 1}^{+ \infty} B_k^* \frac{t^{k - 1}}{k!} = \frac{1}{t} + \sum_{k = 0}^{+ \infty} B_{k + 1}^* \frac{t^k}{(k + 1)!} .
$$
By applying this to the operator $\Delta$ and use the fact that $\log(I + \Delta) = D$, we get the identity of operators:
$$
D^{-1} = \Delta^{-1} + \sum_{k = 0}^{+ \infty} B_{k + 1}^* \frac{\Delta^k}{(k + 1)!} .
$$
Then, by applying this last identity of operators to the polynomial $x\pinf{n - 1}$ (which is of degree $(n - 1)$), we get (according to \eqref{eq19} and \eqref{eq9})
\begin{align*}
\int x\pinf{n - 1} \, d x & = \frac{x\pinf{n}}{n} + \sum_{k = 0}^{n - 1} B_{k + 1}^* (n - 1)\pinf{k} \frac{x\pinf{n - 1 - k}}{(k + 1)!} \\
& = \frac{x\pinf{n}}{n} + \sum_{k = 0}^{n - 1} \frac{1}{n} \binom{n}{k + 1} B_{k + 1}^* x\pinf{n - 1 - k}  \\
& = \frac{x\pinf{n}}{n} + \sum_{k = 1}^{n} \frac{1}{n} \binom{n}{k} B_k^* x\pinf{n - k} \\
& = \sum_{k = 0}^{n} \frac{1}{n} \binom{n}{k} B_k^* x\pinf{n - k} ,
\end{align*}
that is (up to an additive constant)
\begin{equation}\label{eq16}
\int x\pinf{n - 1} \, d x = \sum_{k = 0}^{n} \frac{1}{n} \binom{n}{k} B_{n - k}^* x\pinf{k} .
\end{equation}
Finally, by identifying the coefficients of $x\pinf{k}$ ($1 \leq k \leq n$) in the right-hand sides of \eqref{eq15} and \eqref{eq16}, we derive that for all $k \in \N$, with $k \leq n$, we have
$$
\sum_{k \leq i \leq n} \frac{s(n - 1 , i - 1) S(i , k)}{i} = \frac{1}{n} \binom{n}{k} B_{n - k}^* ,
$$
as required.
\end{proof}

\begin{coll}\label{c5}
For all $n , k \in \N$, with $n \geq k$, we have
$$
\sum_{k \leq i \leq n} \frac{s(n , i) S(i , k)}{i} = (-1)^{n - k} \frac{(n - 1)!}{k!} \sum_{\ell = 0}^{n - k} \frac{(-1)^{\ell}}{\ell!} B_{\ell}^* .
$$
\end{coll}

\begin{proof}
Let $n , k \in \N$ be fixed such that $n \geq k$. Since the formula of the corollary is clearly true for $n = k$, we may assume that $n \geq k + 1$. According to Theorem \ref{t4}, we have for all integer $m \geq k + 1$:
$$
\sum_{k \leq i \leq m} \frac{s(m - 1 , i - 1) S(i , k)}{i} = \frac{1}{m} \binom{m}{k} B_{m - k}^* .
$$
But (according to the recurrence relation \eqref{eq3}), we have for all integers $i , m$ such that $k \leq i \leq m$:
$$
s(m - 1 , i - 1) = s(m , i) + (m - 1) s(m - 1 , i) .
$$
By inserting this into the previous formula, we get (for all integer $m \geq k + 1$)
$$
\sum_{k \leq i \leq m} \frac{s(m , i) S(i , k)}{i} + (m - 1) \sum_{k \leq i \leq m - 1} \frac{s(m - 1 , i) S(i , k)}{i} = \frac{1}{m} \binom{m}{k} B_{m - k}^*
$$
(since $s(m - 1 , m) = 0$). Then, by multiplying by $\frac{(-1)^m}{(m - 1)!}$, we get
$$
\frac{(-1)^m}{(m - 1)!} \sum_{k \leq i \leq m} \frac{s(m , i) S(i , k)}{i} - \frac{(-1)^{m - 1}}{(m - 2)!} \sum_{k \leq i \leq m - 1} \frac{s(m - 1 , i) S(i , k)}{i} = \frac{(-1)^m}{k! (m - k)!} B_{m - k}^*
$$
(which is valid for all integer $m \geq k + 1$). By summing both sides of this last identity from $m = k + 1$ to $m = n$ and remarking that the sum on the left-hand side is telescopic, we derive that:
$$
\frac{(-1)^n}{(n - 1)!} \sum_{k \leq i \leq n} \frac{s(n , i) S(i , k)}{i} - \frac{(-1)^k}{k!} = \frac{1}{k!} \sum_{m = k + 1}^{n} \frac{(-1)^m}{(m - k)!} B_{m - k}^* ,
$$
hence
\begin{align*}
\frac{(-1)^n}{(n - 1)!} \sum_{k \leq i \leq n} \frac{s(n , i) S(i , k)}{i} & = \frac{1}{k!} \sum_{m = k}^{n} \frac{(-1)^m}{(m - k)!} B_{m - k}^* \\
& = \frac{1}{k!} \sum_{\ell = 0}^{n - k} \frac{(-1)^{\ell + k}}{\ell!} B_{\ell}^* ,
\end{align*}
implying the required formula of the corollary.
\end{proof}

\begin{rmk}
In the context of Corollary \ref{c5}, the quantity $(-1)^{n - k} \sum_{\ell = 0}^{n - k} \frac{(-1)^{\ell}}{\ell!} B_{\ell}^*$ represents the coefficient of $t^{n - k}$ in the power series $\frac{t}{(1 + t) \log(1 + t)}$. In fact, we can reprove Corollary \ref{c5} by applying the analytic function $t \mapsto \frac{t}{(1 + t) \log(1 + t)}$ to the operator $\Delta$, in a manner similar to the proof of Theorem \ref{t4}.
\end{rmk}

In another direction, we immediately derive from Theorem \ref{t4} the following corollary providing an expression of the Bernoulli numbers of the second kind in terms of Stirling numbers of the first kind.

\begin{coll}\label{c6}
For all $n \in \N_0$, we have
$$
B_n^* = \sum_{i = 0}^{n} \frac{s(n , i)}{i + 1} .
$$
\end{coll}

\begin{proof}
Giving $n \in \N_0$, it suffices to apply Theorem \ref{t4} for the pair $(n + 1 , 1)$ (instead of $(n , k)$) and note that $S(m , 1) = 1$ (for all $m \in \N$). 
\end{proof}

\begin{rmk}
The formula in Corollary \ref{c6} can be considered an analog of that in Corollary \ref{c3}. Moreover, it can be obtained more simply and independently of Theorem \ref{t4}, by integrating Formula \eqref{eq1} from $X = 0$ to $X = 1$ and using Formula \eqref{eq13}. Later, we will generalize the formula in Corollary \ref{c6} by establishing a closed form for the general sum $\sum_{i = 0}^{n} \frac{s(n , i)}{i + r}$ ($r \in \N$, $n \in \N_0$); see Theorem \ref{t6}.
\end{rmk}

We now proceed to generalize the formula in Corollary \ref{c3} by finding a closed form for the general sum $\sum_{k = 0}^{n} (-1)^k \frac{k!}{k + r} S(n , k)$ ($r \in \N$, $n \in \N_0$).

\begin{thm}\label{t5}
Let $r \in \N$. Then for all $n \in \N_0$, we have
$$
\sum_{k = 0}^{n} (-1)^k \frac{k!}{k + r} S(n , k) = \frac{1}{(r - 1)!} \sum_{k = 0}^{r - 1} \vabs{s(r , k + 1)} B_{n + k} .
$$
\end{thm}

\begin{proof}
Let $r \in \N$ and $n \in \N_0$ be fixed. By multiplying side by side the two polynomial identities
\begin{align*}
\sum_{k = 0}^{n} S(n , k) X\pinf{k} & = X^n \\[-7mm]
\intertext{and} \\[-14mm]
(X + 1) (X + 2) \cdots (X + r - 1) & = \sum_{k = 0}^{r - 1} \vabs{s(r , k + 1)} X^k
\end{align*}
(which are nothing else \eqref{eq2} and \eqref{eq18}), we get
$$
\sum_{k = 0}^{n} S(n , k) (X + r - 1)\pinf{k + r - 1} = \sum_{k = 0}^{r - 1} \vabs{s(r , k + 1)} X^{n + k} .
$$
Then, by applying the operator $\Delta^{-1}$ to the two sides of this last polynomial identity and taking into account \eqref{eq19} and \eqref{eq20}, we get
$$
\sum_{k = 0}^{n} \frac{S(n , k)}{k + r} (X + r - 1)\pinf{k + r} = \sum_{k = 0}^{r - 1} \vabs{s(r , k + 1)} \frac{B_{n + k + 1}(X)}{n + k + 1} + \c_{n , r}
$$
(where $\c_{n , r}$ is a constant depending only on $n$ and $r$). By dividing on $X$ and remarking that $\frac{(X + r - 1)\pinf{k + r}}{X} = (X + r - 1) (X + r - 2) \cdots (X + 1) \cdot (X - 1) (X - 2) \cdots (X - k) =$ \mbox{$(X + r - 1)\pinf{r - 1} \cdot (X - 1)\pinf{k}$} (for all $k \in \N_0$, $k \leq n$), we get
$$
\sum_{k = 0}^{n} \frac{S(n , k)}{k + r} (X + r - 1)\pinf{r - 1} (X - 1)\pinf{k} = \dfrac{\displaystyle\sum_{k = 0}^{r - 1} \vabs{s(r , k + 1)} \frac{B_{n + k + 1}(X)}{n + k + 1} + \c_{n , r}}{X} .
$$
Finally, by letting $X \rightarrow 0$ and remarking that the polynomial $\sum_{k = 0}^{r - 1} \vabs{s(r , k + 1)} \frac{B_{n + k + 1}(X)}{n + k + 1} + \c_{n , r}$ vanishes at $X = 0$ (according to the previous formula), we get
\begin{align*}
\sum_{k = 0}^{n} \frac{S(n , k)}{k + r} (r - 1)\pinf{r - 1} (-1)\pinf{k} & = \lim_{X \rightarrow 0}   \dfrac{\sum_{k = 0}^{r - 1} \vabs{s(r , k + 1)} \frac{B_{n + k + 1}(X)}{n + k + 1} + \c_{n , r}}{X} \\
& = \frac{d}{d X} \left(\sum_{k = 0}^{r - 1} \vabs{s(r , k + 1)} \frac{B_{n + k + 1}(X)}{n + k + 1} + \c_{n , r}\right)(0) \\
& = \sum_{k = 0}^{r - 1} \vabs{s(r , k + 1)} B_{n + k} ~~~~~~~~~~ (\text{according to \eqref{eq11}}) .
\end{align*}
It remains to note that $(r - 1)\pinf{r - 1} = (r - 1)!$ and $(-1)\pinf{k} = (-1)^k k!$ to conclude with the required formula of the theorem.
\end{proof}

\medskip

\begin{expls}~
\begin{itemize}
\item By taking $r = 1$ in Theorem \ref{t5}, we simply obtain the formula of Corollary \ref{c3}.
\item By taking $r = 2$ in Theorem \ref{t5} and taking into account the fact that $\vabs{s(2 , 1)} = \vabs{s(2 , 2)} = 1$, we obtain the remarkable formula:
\begin{equation}\label{eq21}
\sum_{k = 0}^{n} (-1)^k \frac{k!}{k + 2} S(n , k) = B_n + B_{n + 1}
\end{equation}
(valid for all $n \in \N_0$).
\item By taking $r = 3$ in Theorem \ref{t5} and taking into account the facts that $\vabs{s(3 , 1)} = 2$, $\vabs{s(3 , 2)} = 3$, and $\vabs{s(3 , 3)} = 1$, we obtain the formula:
\begin{equation}\label{eq22}
\sum_{k = 0}^{n} (-1)^k \frac{k!}{k + 3} S(n , k) = B_n + \frac{3}{2} B_{n + 1} + \frac{1}{2} B_{n + 2}
\end{equation}
(valid for all $n \in \N_0$). 
\end{itemize}
\end{expls}

We finally present a generalization of the formula in Corollary \ref{c6} by finding a closed form for the general sum $\sum_{k = 0}^{n} \frac{s(n , k)}{k + r}$ ($r \in \N$, $n \in \N_0$).

\begin{thm}\label{t6}
Let $r \in \N$. Then for all $n \in \N_0$, we have
$$
\sum_{k = 0}^{n} \frac{s(n , k)}{k + r} = \sum_{k = 0}^{r - 1} \lambda_{r , n , k} B_{n + k}^* ,
$$
where
$$
\lambda_{r , n , k} := \sum_{\ell = 0}^{r - 1 - k} \binom{r - 1}{\ell} S(r - 1 - \ell , k) n^{\ell}
$$
{\rm(}for all $k \in \{0 , 1 , \dots , r - 1\}${\rm)}.
\end{thm}

\begin{proof}
Let $r \in \N$ and $n \in \N_0$ be fixed. Starting from the polynomial identity in \eqref{eq1}
$$
\sum_{k = 0}^{n} s(n , k) x^k = x\pinf{n} ,
$$
which we multiply by $x^{r - 1}$ and then integrate from $0$ to $1$, we get
$$
\sum_{k = 0}^{n} \frac{s(n , k)}{k + r} = \int_{0}^{1} x\pinf{n} \cdot x^{r - 1} \, d x .
$$
So, we are led to show the identity
$$
\int_{0}^{1} x\pinf{n} \cdot x^{r - 1} \, d x = \sum_{k = 0}^{r - 1} \lambda_{r , n , k} B_{n + k}^* ,
$$
where the $\lambda_{r , n , k}$'s are defined in the statement of the theorem. Using successively the binomial formula and the polynomial identity \eqref{eq2}, we have that
\begin{align*}
\int_{0}^{1} x\pinf{n} \cdot x^{r - 1} \, d x & = \int_{0}^{1} x\pinf{n} (x - n + n)^{r - 1} \, d x \\
& = \int_{0}^{1} x\pinf{n} \left(\sum_{\ell = 0}^{r - 1} \binom{r - 1}{\ell} (x - n)^{r - 1 - \ell} n^{\ell}\right) \, d x \\
& = \sum_{\ell = 0}^{r - 1} \left(\binom{r - 1}{\ell} n^{\ell} \int_{0}^{1} x\pinf{n} (x - n)^{r - 1 - \ell} \, d x\right) \\
& = \sum_{\ell = 0}^{r - 1} \left(\binom{r - 1}{\ell} n^{\ell} \int_{0}^{1} x\pinf{n} \left(\sum_{k = 0}^{r - 1 - \ell} S(r - 1 - \ell , k) (x - n)\pinf{k}\right) \, d x\right) \\
& = \sum_{\ell = 0}^{r - 1} \sum_{k = 0}^{r - 1 - \ell} \binom{r - 1}{\ell} n^{\ell} S(r - 1 - \ell , k) \int_{0}^{1} x\pinf{n} (x - n)\pinf{k} \, d x .
\end{align*}
But for all $k \in \{0 , 1 , \dots , r - 1 - \ell\}$, we have that $x\pinf{n} (x - n)\pinf{k} = x\pinf{n + k}$, so $\int_{0}^{1} x\pinf{n} (x - n)\pinf{k} \, d x = \int_{0}^{1} x\pinf{n + k} \, d x = B_{n + k}^*$ (according to \eqref{eq13}). Thus
\begin{align*}
\int_{0}^{1} x\pinf{n} \cdot x^{r - 1} \, d x & = \sum_{\ell = 0}^{r - 1} \sum_{k = 0}^{r - 1 - \ell} \binom{r - 1}{\ell} n^{\ell} S(r - 1 - \ell , k) B_{n + k}^* \\
& = \sum_{k = 0}^{r - 1} \left(\sum_{\ell = 0}^{r - 1 - k} \binom{r - 1}{\ell} S(r - 1 - \ell , k) n^{\ell}\right) B_{n + k}^* \\
& = \sum_{k = 0}^{r - 1} \lambda_{r , n , k} B_{n + k}^* ,
\end{align*}
as required. This completes the proof.
\end{proof}

\begin{expls}~
\begin{itemize}
\item By taking $r = 1$ in Theorem \ref{t6}, we simply obtain the formula of Corollary \ref{c6}.
\item By taking $r = 2$ in Theorem \ref{t6}, we obtain the formula
$$
\sum_{k = 0}^{n} \frac{s(n , k)}{k + 2} = n B_n^* + B_{n + 1}^*
$$
(valid for all $n \in \N_0$).
\item By taking $r = 3$ in Theorem \ref{t6}, we obtain the formula
$$
\sum_{k = 0}^{n} \frac{s(n , k)}{k + 3} = n^2 B_n^* + (2 n + 1) B_{n + 1}^* + B_{n + 2}^*
$$
(valid for all $n \in \N_0$).
\end{itemize}
\end{expls}

\begin{rmk}
The formula in Theorem \ref{t6} can be considered an analog of the formula in Theorem \ref{t5}; however, in the formula of Theorem \ref{t6}, the coefficients $\lambda_{r , n , k}$ of the linear combination of the Bernoulli numbers of the second kind depend on $n$ (in contrast to the coefficients of the linear combination of the Bernoulli numbers of the first kind in the formula of Theorem \ref{t5}).
\end{rmk}

\bigskip
\hrule
\bigskip

\noindent 2020 {\it Mathematics Subject Classification}:
Primary 11B68; Secondary 11B73.

\noindent \emph{Keywords}: Bernoulli polynomial, Bernoulli number of the first kind, Bernoulli number of the second kind, Stirling number of the first kind, Stirling number of the second kind.

\bigskip
\hrule
\bigskip

\end{document}